\documentclass{tran-l}

\newtheorem{thm}{Theorem}[subsection]

\theoremstyle{definition}
\newtheorem{defn}[thm]{Definition}
\theoremstyle{remark}

\theoremstyle{plain}
\newtheorem{theorem} {Theorem}[section]

\newtheorem{lemma}[theorem]{Lemma}

\theoremstyle{definition}

\theoremstyle{remark}
\newtheorem{remark}[theorem]{Remark}

\newcommand{\be} {\begin{equation}}
\newcommand{\ee} {\end{equation}}
\newcommand{\bea} {\begin{eqnarray}}
\newcommand{\eea} {\end{eqnarray}}
\newcommand{\Bea} {\begin{eqnarray*}}
\newcommand{\Eea} {\end{eqnarray*}}
\newcommand{\pa} {\partial}

\newcommand{\de} {\delta}

\newcommand{\De} {\Delta}

\newcommand{\no} {\nonumber}

\newcommand{\R}{\mathbb R}

\catcode`\@=11

\makeatletter
\@addtoreset{equation}{section}
\makeatother

\begin{document}

\title[]{homoclinic solutions for fourth order traveling wave equations}

\author{Sanjiban Santra, Juncheng Wei }

\address{S. Santra, School of Mathematics and Statistics, The University of Sydney, NSW 2006, Australia.}
\email{sanjiban@maths.usyd.edu.au}
\address{J. Wei, Department of Mathematics, The Chinese University of Hong Kong, Shatin, Hong Kong.}
\email{wei@math.cuhk.edu.hk}



\begin{abstract}  We consider homoclinic solutions of fourth order equations
$$ u^{''''} + \beta^2 u^{''} + V_u (u)=0 \ \ \mbox{in} \ \R\ ,$$
where $V(u)$ is either the suspension bridge type  $V(u)=e^u-1-u$ or Swift-Hohenberg type $ V(u)= \frac{1}{4}(u^2-1)^2$. For the suspension bridge type equation, we prove existence of a homoclinic solution for {\em all} $ \beta \in (0, \beta_*)$ where $ \beta_{*}= 0.7427\cdots$. For the Swift-Hohenberg type equation, we prove  existence of a homoclinic solution for each $\beta \in (0, \beta_{*})$, where $\beta_{*}=0.9342\cdots$. This partially solves a conjecture of Chen--McKenna \cite{YCM1}.

\end{abstract}
\maketitle
\section{Introduction} The study of homoclinic and heteroclinic solutions for  fourth order equations has attracted a lot of attention for the last two decades. Though simple-looking, the fourth order equations appear to be difficult and pose lots of very challenging questions.  We refer to the survey papers \cite{M} and the monograph \cite{PT} for further references.

Motivated by the appearance of traveling wave behavior on the Narrows Tacoma bridge and the Golden Gate bridge, McKenna and Walter \cite{MW1} considered the following nonlinear beam equation
$$w_{tt}+ w_{xxxx}+ V_{w}(w)=0$$ where
$V_{w}$ is the restoring force and is chosen such that the effective force of the cables holds the beam up but the constant force of gravity holds it down  on the assumption that there is no reaction force due to compression. Here $w(x,t)$ denotes the displacement of the beam from the unloaded state \cite{LM}.
This leads to a fourth order beam equation
\be\label{w1}\frac{\pa^2 w}{\pa t^2}+ \frac{\pa^4 w}{\pa x^4} =- w^{+}+ 1 \text{ in } \R\ee  where $w^{+}=\max\{w,0\}.$
Note that (\ref{w1}) also arise in the study of deflection in railway tracks and undersea pipelines. See \cite{AL} and \cite{BH}.\\
If we look for a  traveling wave solution of the type $w(x, t)= 1+ u(x-\beta t),$ then (\ref{w1}) is transformed to a fourth order differential equation of the form
\be\label{1.1} u''''+ \beta^2 u''+ (u+1)^{+}-1=0 \text{ in } \R\ee
where $\beta$ denotes the wave speed.
McKenna--Walter \cite{MW1}  studied (\ref{1.1}) by  solving an ordinary differential equation explicitly as
\begin{equation}
  \label{ws1}
  \left\{
    \begin{aligned}
      u'''' +\beta^2 u''+ u &= 0\; &&\text{if } u\geq -1, \\
 u'''' +\beta^2 u'' &=1 &&\text{if } u\leq -1\\
\end{aligned}
  \right.
\end{equation}
and then glued the two solutions to match at $u=-1$ (called one-trough solutions). In fact, they noticed that as the wave speed approaches $\sqrt{2},$ the solution becomes highly oscillatory in nature, and as $\beta$ approaches $0,$ they appear to go to infinity in amplitude. It was also noticed by numerical experiments that some of the traveling wave solutions appear to be stable,  that is,  when two waves collide, they pass through each other like solitons having many nodes. \\
Later on, Chen-Mckenna \cite{YCM} applied mountain pass theorem on $H^{2}(\R)$ to prove that (\ref{1.1}) has a nontrivial solution. In addition the calculations in \cite{MW1} suggest that there are many solutions,  possibly infinitely many solutions, though it is only known that there exists at least one non-trivial solution. In \cite{CM}, Champneys and McKenna  proved that there exist $0<\beta^{'}<\beta^{''}< \sqrt{2}$ such that (\ref{1.1}) has infinitely many {\em multitroughed} homoclinic solutions for all $\beta\in (\beta^{'}, \beta^{''})$ using the ideas in \cite{BCT}, \cite{D} and \cite{PT3}.\\ The model (\ref{1.1}) has some serious drawbacks. Firstly, it simplifies the nonlinearity of the physical situation, by not allowing nonlinear effects until the deflection is quite large. Secondly,  the non-smoothness of the restoring force leads to numerical difficulties. So the following modified version of (\ref{1.1}) was proposed in \cite{YCM}
\begin{equation}
  \label{a1}
\left\{\begin{aligned}
      u''''+ \beta^2 u''+ e^{u}-1 &= 0 & \text{ in } \R,\\
      u&\not \equiv 0\\
      u &\in H^{2}(\R).
    \end{aligned}
  \right.
\end{equation}
Though the nonlinearity in (\ref{a1}) looks similar to that in (\ref{1.1}),  the study of (\ref{a1}) is quite difficult. In addition, $V(u)= \int_{0}^{u}( e^{t}-1)dt= e^{u}-u-1$ is not symmetric and it has linear growth at $-\infty$ and grows like $e^{u}$ at $+\infty.$\\
In \cite{SW}, Smets--van den Berg used mountain-pass lemma and Struwe's monotonicity trick \cite{St} to prove that  for {\em almost all} $\beta\in (0, \sqrt{2})$, (\ref{a1}) admits a solution.
Later on in \cite{BHMP}, Breuer, Hor\'{a}k, McKenna and Plum  used a computer assisted proof to conclude that if $\beta=1.3,$ there is at least $36$ solutions. It was then conjectured in \cite{BHMP} that there is at least one homoclinic solution for {\em all} $\beta \in (0, \sqrt{2})$.  In this paper, we partially solve this conjecture.

\begin{theorem}\label{th1} There exists $0<\beta_{\star}< 1$ such that for all $\beta\in (0, \beta_{\star}),$ (\ref{a1}) admits a homoclinic solution and $u$ decays in the form $e^{-\tau(\beta) |x|} \cos (ax+ b)$ for some $a, b\in \R$ and $\tau(\beta)>0.$ (Explicitly, $\beta_{\star} \sim 0. 7427\cdots).$
\end{theorem}

 We will also consider the Swift-Hohenberg equation which  is a general model for pattern-forming process derived in \cite{SH} to describe random thermal fluctuations in the Boussinesque equation and in the propagation of lasers \cite{LMN}.  It also arises in the study of ternary mixtures made up of oil, water and surfactant agents yielding
a free energy functional of the Ginzburg-Landau equation given by,
\be\label{ter}\Psi(u)= \int_{\R^3} [ ( \De u)^2 + h(u) |\nabla u|^2+ V(u)]dx\ee
where the scalar parameter $u$ is related to the local difference of the concentration of oil and water \cite{GS}. The function $h$ denotes the amphilic properties and $V(u)$ denotes the potential (the bulk free energy of the ternary mixture) \cite{BM}.  Not only they have important applications in science especially in statistical mechanics of self avoiding surfaces, but also in cell membrane biology, in string theory and in high energy physics \cite{P}. The existence of heteroclinic solution has been studied extensively in \cite{BSTT} when $h$ changes sign. \\
In this paper we also consider
\begin{equation}
\label{e11}
u''''+ \beta^2 u''+ V_{u}(u)=0 \text{ in } \R
\end{equation} where $V_{u}(u)= -u+u^3.$ For this model, a question of interest is phase transition i.e. solutions connecting to $u=\pm 1.$ Peletier and Troy studied homoclinic and heteroclinic solutions  when $h(u)= -\beta^2$ in \cite{PT1}, \cite{PT2}, though nothing is known about the existence of heteroclinic solutions of (\ref{e11}) for $0<\beta<\sqrt{8}.$ Buffoni \cite{B} proved that if $V_{u}(u)=-u+ u^2$, then (\ref{e11}) admits at least one solution for all $\beta\in (0, \sqrt{2}).$\\
Our techniques in proving Theorem \ref{th1} actually allows to conclude similar results for the well-known Swift-Hohenberg model
\begin{equation}
  \label{a2}
\left\{\begin{aligned}
      u''''+ \beta^2 u''+ u(u^2-1) &= 0 &&\text{in } \R\\
      u-1 &\in H^{2}(\R)
      .\end{aligned}
  \right.
\end{equation}
Smets-van den Berg \cite{SW}  proved that for {\em almost all} $\beta \in (0, \sqrt{8})$, problem (\ref{a2}) has a homoclinic solution.  For (\ref{a2}), we have
\begin{theorem}\label{th2} For each  $\beta\in (0, \beta_0)$, where $\beta_0 \approx 0.9342\cdots$,  (\ref{a2}) admits a homoclinic solution.
\end{theorem}
Here $ \beta_0= \sqrt{ \frac{\sqrt{2}}{k_0}} $ where $4 k_0^2 -2k_0 -3=0$. In particular, $\beta_0 \approx 0.9342\cdots.$ As far as we know, Theorems \ref{th1} and \ref{th2} are the first result in establishing the existence of homoclinic solutions for explicit $\beta'$s.

Let us recall some of the difficulties associated to problem (\ref{a1}):

\medskip
\noindent
(a) The functional associated to (\ref{a1}) does not satisfy the global {\em Ambrosetti-Rabinowitz condition}, i.e.
$$ V_{t}(t) t-\theta V(t)\geq 0 \ $$
 for some $ \   \theta > 2 $  and for all $t\in \R.$ This poses a major problem in proving the boundedness of a Palais-Smale sequence in the $H^2$-norm. In \cite{SW}, Struwe's monotonicity trick was used to conclude the $H^2$-boundedness.

\medskip
\noindent
(b) Linearizing  the equation (\ref{a1}) at $u=0$ we obtain
\be\label{lin1} w''''+ \beta^2 w''+ w=0. \ee The roots of (\ref{lin1}) are given by
\be\label{lin2} \mu_{\pm}^2= \frac{-\beta^2\pm \sqrt{\beta^4- 4}}{2}.\ee
Note that if  $\beta\geq \sqrt{2},$ then $\mu_{\pm}^2$ are real and (\ref{a1}) can be written as
$$u''''+\beta^2 u''+ u+ e^{u}-u-1=0$$ and hence can be decomposed into a system
\begin{equation}
  \label{a3}
\left\{\begin{aligned}
      u''-\mu_{+}^2 u  &= w &&\text{in } \R\\
      w''-\mu_{-}^2 w  &= 1- u- e^{u} &&\text{in } \R.
    \end{aligned}
  \right.
\end{equation}
This formulation in fact helps us to obtain a-priori estimates for $u$ and $w$ using strong maximum principle. But if $0<\beta<\sqrt{2}$ we cannot apply this method to reduce to systems, and in fact monotone homoclinics cannot exist in this range.\\

\medskip
\noindent
(c) Let $w= u''.$ Then (\ref{a1}) can be written as $w''+ c w= 1- e^{u}$ where $c=\beta^2>0.$  As a result, we cannot apply maximum principle and we cannot say whether a solution of (\ref{a1}) after a certain stage is positive or negative.\\

\medskip
\noindent
(d) A solution of (\ref{a1}) tends to oscillate infinitely many times even if we have a bound on the Morse index of the solution. This poses a lot of trouble in obtaining solutions converging to zero as $x\rightarrow \pm \infty$.\\

Our main idea of proving Theorem \ref{th1} is to bound the $H^2$ norm  by the energy and the {\em Morse index}.  A crucial tool is the Morse index of the mountain-pass solutions. We believe that a more refined analysis should cover the full range $\beta \in (0, \sqrt{2})$.

Finally let us mention that the idea of using Morse index to bound solutions  has been used in several recent papers for second order elliptic equations. See  Dancer \cite{Da} and Farina \cite{Fa} and  the references therein. However, this idea has never been used in fourth order equations because the Moser iteration does not work.

\medskip\noindent{\bf Notations:} Throughout this paper, by the equality $B =O(A)$ we mean that there exists $C>0$ such
that $|B|\leq C A $ and $C\approx A$ means that $C= A+ o(1)$ where $o(1)\rightarrow 0$ as $\beta\rightarrow 0.$

\section{Preliminaries}
In this section, we prove existence of a mountain-pass solution and show that its Morse index is {\em at most one}. This will be used crucially in the next section.
We first recall the following definition.
\begin{defn} Let $H$ be a Hilbert space and  $B$ be a closed set of $H.$ Let $\mathcal{F}$ be a family of compact subsets of $H.$ Then we call $\mathcal{F}$ a {\em homotopy stable family with boundary $B$} if\\
(a) Every  $A\in \mathcal{F}$ contains $B.$\\
(b) For any $A\in \mathcal{F}$ and any $\eta\in C(H \times [0,1]; H)$ with $\eta(x, t)= x$ and for all $(x, t)\in (H \times\{0\})\cup (B\times [0,1])$ implies that $\eta (A\times [0,1])\in \mathcal{F}.$
\end{defn}
\begin{defn}
A family $\mathcal{F}$ of $G-$ subsets is said to be $G-$ homotopic of  dimension $N$ with boundary $B$ if there exists a compact $G-$ subset $D$ of $\R^N$ containing a closed subset $D_{0}$ and a continuous $G-$ invariant map $\sigma
': D\rightarrow B$ such that
$$\mathcal{F}= \{A\subset H: A= f(D) \text{ for some } f\in C_{G}(D, H) \text{ with } f=\sigma' \text{ on } D_{0}\}.$$
\end{defn}
Define $${K}_{c}= \{u\in H: I(u)=c\, \, ;\langle I'(u), u\rangle =0\}.$$
A Lie group $G$ is said to be a free action if $gx=x$ implies $g=i_{d}$  for any $x\in H$.
We borrow the following lemma from \cite{EG} on page 232.
\begin{lemma}\label{ge} Let $G$ be a compact Lie group acting freely and differentiably  on  $H.$ Let $I$ be a $G-$ invariant functional on $H$ and $\mathcal{F}$ be a $G-$ homotopic of  dimension $N$ stable with boundary $B.$ If  $I$ satisfies $(PS)_{c}$ where $c:= c(I, \mathcal{F})$ and $I''(u)$ is a Fredholm for each level $c$ and $\sup_{B}I< c.$ Then there
exists $u\in K_{c}$ with Morse index of $u$ at most $N.$
\end{lemma}
\begin{proof} For the proof, see \cite{G}, Chapter 10. \end{proof}

We define
\begin{equation}
\label{Ibeta}
I_{\beta}(u)= \frac{1}{2}\int_{\R} |u''|^2dx-\frac{\beta^2}{2} \int_{\R} |u'|^2dx+  \int_{\R} (e^{u}-u-1)dx,  \ \ \ \forall u\in H^{2}(\R).
\end{equation}
First note that $I_{\beta}$ is $C^{2}(H^{2}(\R))$ and it does not satisfy Palais-Smale condition  due to translation invariance of the functional. Moreover, if $u$ is a critical point of $I_{\beta},$ then $u$ is a classical solution of (\ref{a1}).
Also we have $I_{\beta}(0)=0.$
\begin{lemma}\label{n2} There exist $r>0, c>0$ such that $I_{\beta}(u)\geq c\|u\|^2_{H^{2}(\R)}$ for all $u\in B_{r}(0)$ where
$B_{r}$ is a ball centered at the origin in $H^{2}(\R).$ In fact, we can choose $r$ and $c$ to be independent of $\beta.$
\end{lemma}
\begin{proof} If $\|u\|_{H^{2}(\R)}< r,$ then by Sobolev embedding theorem $\|u\|_{L^{\infty}(\R)}< Cr$ for some $C>0$ (independent of $r$). Let us choose $r$ so small   that  if $\|u \|_{H^2 (\R)} <r$, then $(e^{u}-u-1) \geq (\frac{1}{2}-\eta) u^2$ for  some $ \eta = \frac{2-\beta^2}{12}$.
Let $\hat{u}(\xi)$ be a Fourier transform of $u(x).$ Taking Fourier transform we have
\bea I_{\beta}(u)&\geq& \frac{1}{2}\int_{\R} (\xi^4 -\beta^2 \xi^2 + 1) (\hat{u}(\xi))^2 d\xi-\eta\int_{\R} \hat{u}^2(\xi)d\xi\no \\&\geq& \frac{1}{2}\int_{\R} (\xi^4 + \xi^2 + 1- (\beta^2+1)\xi^2) (\hat{u}(\xi))^2 d\xi-
 \eta\int_{\R} \hat{u}^2(\xi)d\xi \no \\&\geq& \frac{1}{2}\int_{\R} \bigg(\xi^4 + \xi^2 + 1- \frac{(\beta^2+1)}{3}(\xi^4+\xi^2+1)\bigg)(\hat{u}(\xi))^2 d\xi \no\\&-& \eta\int_{\R} \hat{u}^2(\xi)d\xi\no \\&\geq& \frac{2-\beta^2}{6}\int_{\R} (\xi^4 + \xi^2 + 1)(\hat{u}(\xi))^2 d\xi-\eta \|u\|^2_{H^{2}(\R)}\no\\&=& \frac{2-\beta^2-6\eta}{6}\|u\|^2_{H^{2}(\R)}=\frac{2-\beta^2}{12}\|u\|^2_{H^{2}(\R)}.
\eea
\end{proof}
\begin{lemma}\label{n3} There exists $e(\text{independent of }\beta)\in H^{2}(\R)$ such that $I_{\beta}(e)<0.$
\end{lemma}
\begin{proof} Choose $v\in H^{2}(\R)$ such that $v$ has compact support and $v<0.$ Let $u_{\lambda}(x)= v(\lambda x), \lambda>0.$ Then it is possible to choose a $\lambda>0$ such that
\be\no \int_{\R} |u_{\lambda}''|^2-\beta^2 |u_{\lambda}'|^2=-\de<0.\ee
For $\{u_{\lambda}<0\}$ we have $e^{u}-u-1> 0.$
Now consider
\bea \frac{I_{\beta} (tu_{\lambda})}{t^2}&=&\frac{1}{2} \int_{\R} (|u_{\lambda}''|^2-\beta^2 |u_{\lambda}'|^2)+  \int_{u_{\lambda}<0}\frac{e^{tu_{\lambda}}-tu_{\lambda}-1}{t^2}dx\no\\&= & -\frac{\de}{2}+ \int_{u_{\lambda}<0}\frac{e^{tu_{\lambda}}-tu_{\lambda}-1}{t^2}dx.\eea
But note that the second term is an integral over a bounded domain as support of $u$ is compact
and $\frac{e^{tu_{\lambda}}-tu_{\lambda}-1}{t^2}\rightarrow 0$ as $t \rightarrow \infty.$
This implies that
$$I_{\beta}(t u_{\lambda}) \rightarrow -\infty   \text{ as } t \rightarrow +\infty.$$ Hence the result follows.
\end{proof}
Choose $e= t u_{\lambda}.$ From \cite{SW} we know that for almost all $\beta\in (0, \sqrt{2}),$ $I_{\beta}$ satisfies Palais-Smale condition and hence there exists a mountain pass critical value $c_{\beta}$ and $$c_{\beta}= \inf_{\gamma\in \Gamma} \max_{t\in [0,1]}I_{\beta}(\gamma(t))>0$$
where $$\Gamma= \{\gamma: \gamma\in C([0,1], H^2(\R)); \gamma(0)=0, \gamma(1)=e \}$$ and $I_{\beta}(e)<0.$
\begin{lemma}\label{mpcb} There exists a constant $C>0$ independent of $\beta$ such that $c_{\beta}\leq C.$
\end{lemma}
\begin{proof} We have $I_{\beta}(e)<0.$ Define a path $\gamma: [0,1] \rightarrow H^2(\R)$ such that $\gamma_{1} (\bar{t})= \bar{t}e.$ Then
$$c_{\beta}\leq \max_{\bar{t}\in [0,1]} I_{\beta} (\gamma_{1}(\bar{t}))\leq C.$$ Hence $c_{\beta}$ is uniformly bounded.
\end{proof}
\begin{remark}\label{morse} Also $I_{\beta}''(u)$ can be expressed as $I_{d}-K$ where $I_{d}$ is the identity map and $K$ is a compact operator. Let $G= \{i_{d}\};$ the trivial group consisting of the identity element.
$I_{\beta}$ is a $G-$ invariant functional which satisfies Palais Smale condition \cite{SW}, for almost all $\beta\in (0, \sqrt{2}).$
 Choose $B=\{0,e\}$ and let $\mathcal{F}_{0}^{e}$ be the collection of all paths joining $0$ and $e.$ Then $\mathcal{F}_{0}^{e}$ is a homotopy stable family with boundary $B.$  Moreover, $\sup_{B} I_{\beta}< c_{\beta}.$ Hence by Lemma \ref{ge}, the solution $u_{\beta}$ found in \cite{SW} has Morse index at most one. Also note that $c_{\beta}$ is a decreasing function of $\beta.$\end{remark}

We summarize the results in the following theorem

\begin{theorem}\label{th2.1}
For almost all $ \beta \in (0, \sqrt{2})$, there exists a mountain-pass solution $ u_\beta$ of (\ref{a1}) such that

(1) $ 0 < c_\beta= I_\beta (u_\beta) < C$, where $C$ is independent of $\beta \in (0, \sqrt{2})$,

(2) $ u_\beta$ has Morse index at most one and $ u_\beta \in H^2 (\R)$,

(3) the following pointwise identity holds
\be\label{poho} u'(x) u'''(x)- \frac{(u''(x))^2}{2}+\frac{\beta^2}{2} (u'(x))^2+ e^{u(x)}-u(x)-1=0.\ee
\end{theorem}

We will call (\ref{poho}) a kind of Pohozaev identity which  follows by multiplying (\ref{a1}) with $u'$ and then integrating in $(-\infty, x)$.

\section{Key Inequalities}
In this section, we prove the following key inequalities which will be used to bound the part where $u$ is large.
\begin{lemma} \label{l1}
Let $k_1>1$ be such that
\begin{equation}
k_1^2 -1 -k_1 -\sqrt{k_1^2-1} =0.
\end{equation}
Then we have
\begin{equation}
\label{K1}
\int_{-a}^a (u^{''})^2-\beta^2 \int_{-a}^a (u^{'})^2 + \frac{k_1^2 \beta^4 }{4} \int_{-a}^a u^2 \geq 0
\end{equation}
for all $ u \in H^2 (-a, a)$ and $ u(-a)= u(a)$.
\end{lemma}

\begin{proof} First, we note that  the following inequality holds
\begin{equation}
\label{K1-1}
\int_{-a}^a (u^{''})^2-\beta^2 \int_{-a}^a (u^{'})^2 + \frac{ k^2 \beta^4 }{4} \int_{-a}^a u^2 \geq 0
\end{equation}
for all $k\geq 1,   u \in H^2 (-a, a), u(-a)= u(a)=0$. See Lemma 5 of \cite{DB}.

Hence we may assume that $ u(-a)=u(a) \not =0 $. Then by rescaling, we may assume that $ \beta^2=2 $ and $ u(\pm a)=1$. We consider the following minimization problem
\be \label{K1-2}
 M_a^{}= \min_{ u \in H^2 (-a, a), u(\pm a)=1 } \int_{-a}^{a} (u^{''})^2 -2 \int_{-a}^{a}(u^{'})^2 + k^2 \int_{-a}^{a} u^2.\ee
Using the inequality (\ref{K1-1}), it is easy to see that the minimizer in (\ref{K1-2}) exists and satisfies
\begin{equation}
  \label{ki1}
\left\{\begin{aligned}
      u''''+ 2 u''+ k^2 u  &= 0 &&\text{in } (-a, a),\\
      u(\pm a) = 1 , u''(\pm a)&=0.
    \end{aligned}
  \right.
\end{equation}
We can assume that $u$ is even since by (\ref{K1-1}),  the solution to
\begin{equation}
  \label{ki1-2}
\left\{\begin{aligned}
      u''''+ 2 u''+ k^2 u  &= 0 &&\text{in } (-a, a)\\
      u(\pm a) = 0 , u''(\pm a)&=0
    \end{aligned}
  \right.
\end{equation}
is zero, if $ k>1$. From (\ref{ki1}), we conclude
\be u(x)= A\cosh \lambda x \cos\mu x + B \sinh \lambda x \sin\mu x \ee
where $r= \lambda + i\mu $ are the roots of $r^4+ 2 r^2+k^2=0.$
Then the minimum can be computed \be\label{ki2}  M_a= 2\bigg[\int_{0}^a (u^{''})^2 -2 \int_{0}^a (u^{'})^2 + k^2 \int_{0}^a u^2\bigg]= -2 u(a) (u'''(a)+ 2 u'(a)). \ee
In order to show that $M_a \geq 0$ we proceed to calculate $A, B.$

First we compute the derivatives of $u$;
$$u'(x)= (\lambda A+ \mu B) \sinh \lambda x \cos\mu x+ ( B\lambda -\mu A) \cosh \lambda x \sin\mu x,$$
\bea \no u''(x)&=& ((\lambda^2- \mu^2) A+ 2\lambda\mu B) \cosh \lambda x \cos\mu x+ ( B(\lambda^2-\mu^2) -2\mu \lambda A) \sinh \lambda x \sin\mu x,\eea
\bea \no u'''(x) &=& (\lambda(\lambda^2- \mu^2) A+ 2\lambda^2\mu B+\mu B(\lambda^2-\mu^2)-2 A \lambda \mu^2) \sinh \lambda x \cos\mu x\\\no&+& ( \lambda B(\lambda^2-\mu^2) -2 A \lambda^2\mu-2\mu (\lambda^2-\mu^2) A- 2B \lambda \mu^2) \cosh \lambda x \sin\mu x.\eea
Using $u''(a)=0$ we have
\bea u''(a)&=& ((\lambda^2- \mu^2) A+ 2\lambda\mu B) \cosh \lambda a \cos\mu a\no \\&+& ( B(\lambda^2-\mu^2) -2\mu \lambda A) \sinh \lambda a \sin\mu a\eea
and from $u(a)=1$ we get
\begin{equation}
  \label{ki11}
\left\{\begin{aligned}
      A \cosh \lambda a \cos\mu a+ B \sinh \lambda a \sin\mu a &= 1 \\
      2B \lambda \mu \cosh \lambda a \cos\mu a- 2 A \lambda \mu \sinh \lambda a \sin\mu a&= (\mu^2-\lambda^2)
    \end{aligned}
  \right.
\end{equation}
which implies that
\be A \cosh \lambda a \cos\mu a+ B \sinh \lambda a \sin\mu a = 1\ee
\be - A \sinh \lambda a \sin\mu a+ B  \cosh \lambda a \cos\mu a=\frac{\mu^2-\lambda^2}{2\lambda \mu}.\ee
By simple computations  we obtain
$$A=\frac{\cosh \lambda a \cos\mu a- \frac{\mu^2-\lambda^2}{2\lambda \mu} \sinh \lambda a \sin \mu a}{\cosh^2 \lambda a \cos\mu^2 a+ \sinh^2 \lambda a \sin^2 \mu a }$$ and
$$B =\frac{\sinh \lambda a \sin\mu a+ \frac{\mu^2-\lambda^2}{2\lambda \mu} \cosh \lambda a \cos \mu a}{\cosh^2 \lambda a \cos\mu^2 a+ \sinh^2 \lambda a \sin^2 \mu a }.$$
Now
\bea\label{z1} u'''(a)+ 2 u'(a) &=& A (\lambda(\lambda^2-\mu^2)- 2\lambda \mu^2+ 2\lambda) \sinh \lambda a \cos \mu a \no \\&+&
B (\mu (\lambda^2-\mu^2)+ 2\lambda^2\mu+ 2\mu) \sinh \lambda a \cos \mu a\no \\&+&
A (-2\lambda^2\mu-\mu (\lambda^2-\mu^2)-  2\mu) \cosh \lambda a \sin \mu a\no \\&+&
B (-2\lambda\mu^2+\lambda (\lambda^2-\mu^2)+  2\lambda) \cosh \lambda a \sin \mu a.\eea
As a result we have
\bea&& u'''(a)+ 2 u'(a)\\\no &=&\frac{1}{4\lambda}[2\lambda^2(\lambda^2-\mu^2)- 4 \lambda^2\mu^2+ 4\lambda^2 + (\lambda^2-\mu^2+ 2\lambda\mu+2)(\mu^2-\lambda^2)]\sinh 2\lambda a \\\no &+&\frac{1}{4\mu}[ (\lambda^2-\mu^2- 2\lambda\mu+2)(\mu^2-\lambda^2)- 2 \mu (2\lambda^2\mu +\mu(\lambda^2-\mu^2)+ 2\mu)]\sin 2\mu a \\\no &=&\frac{1}{4\lambda}[(\lambda^2+ \mu^2)(\lambda^2-\mu^2)- 4 \lambda^2\mu^2+ 4\lambda^2 + (2 \mu \lambda+2)(\mu^2-\lambda^2)]\sinh 2\lambda a \\\no &-&\frac{1}{4\mu}[ (\lambda^2+ \mu^2)(\lambda^2-\mu^2)- 4 \lambda^2\mu^2+ 4\mu^2 + (2 \mu \lambda-2)(\mu^2-\lambda^2)]\sin 2\mu a.\eea
Since $r= \lambda+ i\mu$ is a root of $r^{4}+ 2r^{2}+ k^2=0, $  we have
\be\label{z2} (\lambda^2-\mu^2)^2- 4\lambda^{2}\mu^2+ 2(\lambda^2-\mu^2)+ k^2=0.\ee
Let $(\lambda^2-\mu^2)=-1$. Then from (\ref{z2}) we have $4\lambda^2 \mu^2= k^2-1$ and hence
$(\lambda^2+ \mu^2)^2= k^2.$ Hence from (\ref{z1}) we have
\be \label{z3} -(u'''(a)+ 2 u'(a))=- \frac{1}{4}(k-k^2+1+ \sqrt{k^2-1})\bigg[\frac{1}{\lambda}\sinh 2\lambda a-\frac{1}{\mu} \sin 2\mu a\bigg].\ee
Now we determine the sign of $(\sinh 2\lambda a-\frac{\lambda}{\mu} \sin 2\mu a)$.
Note that we have $\lambda^2+ \mu^2=k$ and  $\lambda^2- \mu^2=-1.$ Hence we have
$\mu=\sqrt{\frac{k+1}{2}}$ and $\lambda=\sqrt{\frac{k-1}{2}}.$\\
Let $x= 2\mu a.$ Then we have
$$ \sinh 2\lambda a-\frac{\lambda}{\mu} \sin 2\mu a =\sinh \frac{\lambda }{\mu} x- \frac{\lambda }{\mu}\sin x.$$
But we know that
$$\sinh \frac{\lambda }{\mu} x> \frac{\lambda }{\mu}x> \frac{\lambda }{\mu}\sin x \, \,  \forall\, \,  x. $$
Hence
$$M_a=-2u(a)(u'''(a)+ 2 u'(a))\geq 0$$ provided $k^2-k-1-\sqrt{k^2-1}\geq 0, k>1.$
This proves (\ref{K1}).
\end{proof}

\begin{lemma}\label{l2}
Let $k_2$ be such that
\begin{equation}
4 k_2^2 -2 k_2 -3=0, k_2 >1.
\end{equation}
Then we have
\begin{equation}
\label{K2}
\int_{a}^{\infty} (u^{''})^2-\beta^2 \int_{-a}^{\infty} (u^{'})^2 + \frac{k_2^2 \beta^4 }{4} \int_{a}^{\infty} u^2 \geq 0
\end{equation}
for all $ u \in H^2 (a, +\infty)$.
\end{lemma}
\begin{proof}  As before, we may assume that $a=0, u(0)=1, \beta^2=2$. Using the inequality (\ref{K1-1}), it is easy to see that  the minimizer  for the problem
\be
\label{M1.1}
 M_a^{'}= \min_{ u \in H^2 (0, \infty), u(0)=1 } \int_{0}^\infty (u^{''})^2 -2 \int_{0}^\infty (u^{'})^2 + k^2 \int_{0}^\infty u^2
\ee
exists and satisfies
\begin{equation}
  \label{ki1-3}
\left\{\begin{aligned}
      u''''+ 2 u''+ k^2 u  &= 0 &&\text{in } (0, +\infty)\\
      u(0) = 1 , u''( 0)&=0.
    \end{aligned}
  \right.
\end{equation}
Hence
\be u(x)= Ae^{-\lambda x}\cos\mu x + B e^{-\lambda x} \sin\mu x \ee
where $r= -\lambda + i\mu $ are the roots of $r^4+ 2 r^2+k^2=0.$ Similar computations as in Lemma \ref{l1} give
$$u'''(0)+ 2 u'(0)\geq 0$$ if
$4k^2-2k-3\geq  0, k>1.$
\end{proof}
\begin{remark}\label{r1} Similar results as Lemma \ref{l2} holds  for $u\in H^2(-\infty, -a).$ Define $k_0 = \max\{k_1, k_2\}$. It is easy to see that $k_0=k_1 \approx 1.62\cdots$.
\end{remark}

\section{proof of Theorem \ref{th1}}
From Theorem \ref{th2.1},  we have \be\label{up}0< I_{\beta}(u_{\beta})= \frac{1}{2}\int_{\R} |u_{\beta}''|^2dx-\frac{\beta^2}{2} \int_{\R} |u_{\beta}'|^2dx+  \int_{\R} (e^{u_{\beta}}-u_{\beta}-1)dx < C\ee  where $C>0$ is independent of $\beta.$

Let $\beta \in (0, \sqrt{2})$ be fixed. By Theorem \ref{th2.1}, there exists a sequence $\beta_n \to \beta$ and a sequence of   solutions of (\ref{a1}), called $ u_{\beta_n}$, with Morse index at most one  and the bound (\ref{up}). Our main idea is to show that the limit of $ u_{\beta_n}$ exists and has uniform $H^2$ bound.

We will drop the subscript $\beta$ for the sake of convenience.

By simple computations,  it is easy to check that the function
$\frac{e^{u}-u-1}{u^2}$ is increasing if   $ u <0$.  Thus there is a unique  $u_{\star}<0$ such that
\be \label{ustar}
\frac{e^{u_{\star}}-u_{\star}-1}{u_{\star}^2}=\frac{\beta^4k_{0}^2}{8}.\ee
Since $ e^{u}-1-u \geq \frac{u^2}{2}$ for $u>0$, we deduce that
\be \frac{e^{u}-u-1}{u^2}\geq \frac{\beta^4 k_{0}^2}{8} \text { for } u\geq u_{\star}\ee
if
\begin{equation}
\label{beta1}
\beta \leq \frac{\sqrt{2}}{\sqrt{k_0}}.
\end{equation}
Also  we have
\begin{equation}
\label{negative}
e^{u_{\star}}\geq 1+ u_{\star}+ \frac{u_{\star}^2}{2}e^{u_{\star}}.
\end{equation}
This implies that $$e^{u_{\star}}\leq \frac{\beta^4k_0 ^2}{4}.$$

Our main idea is to  bound the energy on the level sets $\{ u \geq u_{\star} \}$ and $\{ u \leq u_{\star} \}$.   On the set $\{ u \geq u_{\star} \}$, we use the key inequality (\ref{K1}). On the set $\{ u \leq u_{\star} \}$, we will use Morse index information obtained in Theorem \ref{th2.1}.\\
First, as a result of Remark \ref{r1} and the key inequalities (\ref{K1}) and (\ref{K2})  we have
\bea &&\frac{1}{2}\int_{A} |u''|^2dx-\frac{\beta^2}{2} \int_{A} |u_{}'|^2dx+  \int_{A} (e^{u}-u-1)dx\no\\&\geq & \frac{1}{2}\int_{A} |u_{}''|^2dx-\frac{\beta^2}{2} \int_{A} |u'|^2dx+  \frac{\beta^4 k_0^2}{4}\int_{A} \frac{u^2}{2} dx \geq 0\eea
where $A=\{u\geq u_{\star}\}.$ (Note that since $ u$ is a homoclinic, $ A= (-\infty, b_0)\cup \displaystyle{\cup_{j=1}^{l} (a_j, b_j)}\cup (a_{l+1}, +\infty)$.)

Our main objective is then to  show that in the complement of $A^{c}=\{u\leq u_{\star}\},$
\be\label{m1} \frac{1}{2}\int_{A^{c}} |u''|^2dx-\frac{\beta^2}{2} \int_{A^{c}} |u_{}'|^2dx+  \int_{A^{c}} (e^{u_{}}-u_{}-1)dx\geq 0.\ee
Let $A^{c}=\{u\leq u_{\star}\}=\cup_{j=1}^{m}(a_{j}, b_{j})$ where $m$ is finite since $u$ is homoclinic.  Since Morse index of $u$ is at most one, then except at most one interval $(a_{i}, b_{i})$ we must have, for $j\neq i$
\be\label{m2} \int_{a_{j}}^{ b_j} |\varphi''|^2dx-\beta^2 \int_{a_{j}}^{ b_j} |\varphi'|^2dx+  \int_{a_{j}}^{ b_j} e^{u}\varphi^2 dx\geq 0 \, \,\,  \forall \varphi\in C^{2}_{0} (a_j, b_j).\ee
Without loss of generality let (\ref{m2}) hold in  some interval $(a, b).$

As $e^{u}$ is an increasing function, in $A^{c}$,  we have
\begin{equation}
\label{n1}
e^{u}\leq e^{u_{\star}}.
\end{equation}
Note that from (\ref{m2}) we have
\be\label{m3} \int_{a }^{ b} |\varphi''|^2dx-\beta^2\int_{a}^{ b} |\varphi'|^2dx+  e^{u_{\star}}\int_{a}^{ b} \varphi^2 dx\geq 0 \, \,\,  \forall \varphi\in C^{2}_{0} (a, b).\ee

(\ref{m3}) implies that the length of the interval $(b-a)$ can be controlled. In fact, we will have
\begin{equation}
\label{astar}
\frac{b-a}{2} \leq a_{\star}
\end{equation}
where $a_{\star}$ depends on $\beta$.\\
On the other hand, if $v=u-u_{\star},$ then we have
\be\label{res} v''''+ \beta^2 v''+ e^{v+ u_{\star}}-1=0.\ee
Multiplying by $v$ and integrating (\ref{res}) we obtain
\be \label{res2} -v'v''\mid_{a}^{b}+\int_{a}^{b} (v'')^2-\beta^2\int_{a}^{b} (v')^2+\int_{a}^{b}(e^{u_{\star}+ v}-1)v=0.\ee

Integrating (\ref{poho}) we have
\be \label{poho1-1} v'v''\mid_{a}^{b}-\frac{3}{2}\int_{a}^{b} (v'')^2+\frac{\beta^2}{2}\int_{a}^{b} (v')^2+\int_{a}^{b}(e^{u_{\star}+ v}-1-u_{\star}-v)=0.\ee
Adding (\ref{res2}) and (\ref{poho1-1}) we obtain that
\be \label{id1} \int_{a}^{b} (v'')^2+ \beta^2 (v')^2= 2\int_{a}^{b}(e^{u_{\star}+ v}-1)v+ 2 \int_{a}^{b}(e^{u_{\star}+ v}-u_{\star}-v-1)dx.\ee
Substituting (\ref{id1}) into the energy over $(a,b)$ we have
\bea\label{rese} I_{\beta}\mid_{(a, b)}(u)&=& \frac{1}{2}\int_{a}^{b} (u'')^2- \frac{\beta^2}{2}\int_{a}^{b} (u')^2+ \int_{a}^{b} (e^u-u-1)dx\no\\&=&\int_{a}^{b} (v'')^2- \int_{a}^{b}(e^{u_{\star}+ v}-1)v.\eea
Now we claim that
$$I_{\beta}\mid_{(a, b)}(u)\geq 0.$$
We argue by contradiction. If not, then $I_{\beta}\mid_{(a, b)}(u)\leq 0.$
We have
\be \label{1.19} \int_{a}^{b} (v'')^2\leq \int_{a}^{b}(e^{u_{\star}+ v}-1)v\leq \int_{a}^{b}(-v).\ee
Without loss of generality we can consider $(a,b)$ to be $(-a, a)$.
Let
\[ A=\int_{-a}^{a} (e^{u_{\star}+v}-1) v, \ B= \int_{-a}^{a} (e^{u_{\star}+v}-1- u_{\star}-v), \  \sigma:= \frac{A}{B}\]
Then from (\ref{id1})  we have
\be \int_{-a}^{a} (u'')^2+ \beta^2 \int_{-a}^{a} (u')^2\leq 2\bigg(1+ \sigma \bigg) \int_{-a}^{a} (e^{v+ u_{\star}}-u_{\star}-v-1).\ee
Thus we have from (\ref{rese})
\bea\label{1.27-1} I_{\beta}\mid_{(-a,a)}(u)&\geq& \frac{1}{2}\bigg\{ \bigg(1+ \frac{1}{1+ \sigma }\bigg)\int_{-a}^{a} (u'')^2-\beta^2 \bigg(1- \frac{1}{1+\sigma}\bigg)\int_{-a}^{a} (u')^2  \bigg\}\no\\&\geq & \frac{1}{2}\frac{1}{(1+ \sigma)}\bigg\{ \bigg(2+ \sigma\bigg)\int_{-a}^{a} (u'')^2- \sigma \beta^2 \int_{-a}^{a} (u')^2  \bigg\}\\\no&\geq& c_{0} \bigg\{  \int_{-a}^{a} (u'')^2-  \frac{\sigma}{2+\sigma}  \beta^2\int_{-a}^{a} (u')^2 \bigg\}.\eea
As a consequence, if $I_{\beta}\mid_{(-a,a)}(u)\leq 0,$ then from (\ref{1.27-1}) we obtain that
\be \label{1.28-1}  \int_{-a}^{a} (u'')^2\leq \frac{\sigma}{2+\sigma} \beta^2 \int_{-a}^{a} (u')^2dx .\ee

Consider the following eigenvalue problem
\begin{equation}
  \label{1.29}
  \left\{
    \begin{aligned}
    u''''+ \lambda^2 u''&= 0\; &&\text{in } (-a, a)\\
 u= u''&= 0 &&\text{ on } \pa (-a, a).
\end{aligned}\right.
\end{equation}
By (\ref{1.28-1}) the first eigenvalue $\lambda_1^2$  of (\ref{1.29}) satisfies $\lambda_1^2\leq \frac{\sigma  \beta^2}{2+\sigma }.$ But note that $u= A\cos (\lambda_1 x),$ with
$\cos \lambda_1 a=0$ implies that $ \lambda_{1} a=\frac{\pi}{2}.$ This implies that $\lambda_1^2 a^2= \frac{\pi^2}{4}$
and hence $\frac{\pi^2}{4}\leq \frac{\sigma}{2+\sigma}  \beta^2 a^2.$\\
As a result we obtain that
\be\label{1.30-1} \beta a \geq \frac{\pi}{2} \sqrt{1+\frac{2}{\sigma}}.\ee
To estimate $\sigma$, let us notice that
\be\label{zs} \inf_{H^{2}(I)\cap H^{1}_{0}(I)} \frac{\int_{-a}^{a} (v'')^2}{(\int_{-a}^{a} v)^2 }=\frac{15}{4 a^{5}}.\ee
In order to prove (\ref{zs}) consider the problem
\begin{equation}
  \label{m5}
  \left\{
    \begin{aligned}
    u''''&= 1\; &&\text{in } (-a, a) \\
 u= u''&= 0 &&\text{ on } \pa (-a, a).
\end{aligned}\right.
\end{equation}
Then $$u(x)=\frac{1}{24}(x^{4}- a^{4})-\frac{1}{4}a^2(x^{2}- a^2)$$ and as a result we have
\be \int_{-a}^{a} u dx=\frac{1}{24}(6+ \frac{2}{5})a^{5}=\frac{4}{15} a^5\ee and
\be \int_{-a}^{a} (u'')^2 dx=\frac{1}{24}(6+ \frac{2}{5})a^{5}=\frac{4}{15} a^5.\ee
Hence we have
\be \label{1.21}\int_{-a}^{a} (v'')^2dx \geq \frac{15}{4 a^{5}}\bigg(\int_{-a}^{a} |v|\bigg)^2 .\ee
But from (\ref{1.19}) we have
\be \label{1.22} \int_{-a}^{a} (v'')^2dx \leq \int_{-a}^{a} (-v) \leq \int_{-a}^{a} |v|dx.\ee
This implies that
\be
\label{1.23}
A \leq  \int_{-a}^{a} |v|dx \leq \frac{4}{15} a^{5}.\ee
But
\be\no B \geq A+ (-1- u_{\star})2a \ee
and hence from (\ref{astar}) we have,
\be
\label{sigma}
\frac{1}{\sigma} \geq 1+ \frac{ 15(-1-u_{\star})}{2 a_{\star}^4}.
\ee
As a result of (\ref{1.30-1}) we have
\be \label{z5}\frac{\sqrt{\frac{15 (-1-u_{\star})}{a^{4}_{\star}}+ 3}}{2\beta}\pi\leq a_{\star}.\ee
So as long as
\be\label{1.44-1} \frac{\sqrt{\frac{15 (-1-u_{\star})}{a^{4}_{\star}}+ 3}}{2}\pi> \beta a_{\star} \ee
holds we have a contradiction with (\ref{astar}).\\
Next we show that condition (\ref{1.44-1}) holds when $\beta$ is small.
In fact, we have from (\ref{ustar}) that for small $\beta,$
\be\label{1.6} u_{\star}\approx-\frac{8}{\beta^4k_0^2}\ee
and hence $ e^{u_{\star}} \approx e^{-\frac{8}{\beta^4 k_0^2}} \approx 0$,  we may assume that $ e^{u_\star} =0$.
Hence we solve the eigenvalue problem
\begin{equation}
  \label{m41}
  \left\{
    \begin{aligned}
     \varphi^{''''}+ \beta^2 \varphi''  &= 0\; &&\text{in } (-a, a) \\
 \varphi= \varphi'&= 0 &&\text{ on } \pa (-a, a).\\
\end{aligned}\right.
\end{equation}
Any even eigenfunction of (\ref{m41}) is $1+ \cos \beta x$ provided that $\beta a= \pi.$
Hence from (\ref{m3}) and (\ref{m41}) we obtain
\be\label{1.11}
a \leq a_{\star}= \frac{\pi}{\beta} + O(\beta)
\ee
and since $ u_{\star} \approx -\frac{8}{\beta^4 k_0^2}$, we obtain
\be
\label{sigma1}
3+ \frac{ 15(-1-u_{\star})}{ a_{\star}^4} \geq  3+ \frac{120}{k_0^2 \pi^4} + O(\beta) \geq \frac{9}{2} + O(\beta)
\ee
which implies that (\ref{1.44-1}) holds for $\beta$ small.

We have thus proved that in $A^c$, except one interval,
\be I_{\beta} \mid_{(a, b)}\geq 0.\ee

Let $(a, b)$ be the exceptional interval in $A^c$. Then we  have
\be \beta (b-a)< 4\pi.\ee
In fact,  if $\beta (b-a)\geq 4\pi,$ then we can construct $\psi_1$ and $\psi_2$ having disjoint support such that
$$\left\{
    \begin{aligned}
   \psi_1(x)&= \cos \beta x \; &&\text{in } (-\frac{\pi}{2\beta}, \frac{\pi}{2\beta}) \\
 \psi_1(x)&= \psi_1''(x)= 0 &&\text{ on } \pa (-\frac{\pi}{2\beta}, \frac{\pi}{2\beta}),
\end{aligned}\right.
$$
$$\left\{
    \begin{aligned}
   \psi_2(x)&= \cos \beta x \; &&\text{in } (\frac{\pi}{2\beta}, \frac{3\pi}{2\beta}) \\
 \psi_2(x)&= \psi_2''(x)= 0 &&\text{ on } \pa (\frac{\pi}{2\beta}, \frac{3\pi}{2\beta})
\end{aligned}\right.
$$
and $\psi_1$ and $\psi_2$ contribute two to the Morse index of $u$, a contradiction to Theorem \ref{th2.1}.

From (\ref{id1}), we have
\be \label{id1-1} \int_{a}^{b} (v'')^2+ \beta^2 (v')^2 \leq C+ C \int_{a}^b |v|
\ee
which yields
\be
\label{1.32} |u| \leq |v|+ |u_{\star}| \leq C \text{ in } (a, b).\ee

Then from (\ref{up}) we have
\be
-C\leq  I_{\beta} \mid_{(a, b)}(u)
\leq C. \ee

Let $A^{'}= A \backslash (a, b)$. Then
\be\label{1.33} 0< \int_{A^{'}} (u'')^2 -\beta^2 \int_{A^{'}} (u')^2+ \int_{A^{'}} (e^{u}-u-1)dx \leq C\ee
and this implies that
\be\label{1.34}  \int_{A^{'}} (e^{u}-u-1)dx + \int_{A^{'}} ((u'')^2+ (u')^2+ u^2) \leq C\ee
and hence $|u|\leq C$ in $A$.
Multiplying (\ref{a1}) by $u$ and integrating we obtain
\be\label{1.35} \int_{\R} (u'')^2-\beta^2 \int_{\R} (u')^2+ \int_{\R}(e^{u}-1) u=0 .\ee
From (\ref{up}) and (\ref{1.35}) we have
\be \label{1.36}-\frac{1}{2}\int_{\R} (e^{u}-1) u+ \int_{\R} (e^{u}-u-1) dx< C\ee
and this implies
\be \label{1.37} \int\limits_{u<0}[ (e^{u}-u-1)-\frac{1}{2} (e^{u}-1) u] < C.\ee
Moreover, using (\ref{id1}) we obtain
\be \int_{\R} (u'')^2+ \beta^2 \int_{\R} (u')^2= 2\int_{\R}(e^{u}-1) u+ 2\int_{\R}(e^{u}-u-1) \leq C.\ee
This implies
\be \|u_{\beta}\|_{H^{2}(\R)}\leq C.\ee
Let $\beta\in (0, \sqrt{2})$ such that there exists $\beta_{n}\rightarrow \beta$ as $n\rightarrow \infty$ and
for $\beta=\beta_{n}$ (\ref{a1}) has a solution. Hence we have
$$u_{\beta_{n}}''''+ \beta_{n}^2 u_{\beta_{n}}''+ e^{u_{\beta_{n}}}-1=0.$$
Also we have $\|u_{\beta_{n}}\|_{H^{2}(\R)}\leq C$ and hence $u_{\beta_{n}}\rightharpoonup u_{\beta}$ in $H^{2}(\R)$ and as a
result, we have $u_{\beta_{n}}\rightarrow u_{\beta}$ in $L^{p}_{loc}(\R)$ as $n\rightarrow +\infty$ for all $p.$ In particular, $u_{\beta_{n}}\rightarrow u_{\beta}$ in $C^{1}_{loc} (\R).$   Hence $u_{\beta_{n}}(x)\rightarrow u_{\beta}(x)$ pointwise almost everywhere. Thus  $e^{u_{\beta_{n}}(x)}\rightarrow e^{u_{\beta}(x)}$ almost everywhere. As $n\rightarrow +\infty,$ we have
$$u_{\beta}''''+ \beta ^2 u_{\beta }''+ e^{u_{\beta }}-1=0.$$
Finally we prove that $u_{\beta}$ is nontrivial, i.e., $ u_{\beta} \not \equiv 0$.  As $u_{\beta_{n}}$ is a solution of  (\ref{a1}) from (\ref{1.35}) we have
\be \int_{\R} (u_{\beta_{n}}'')^2-\beta_{n}^2 \int_{\R} (u_{\beta_{n}}')^2=- \int_{\R} (e^{u_{\beta_{n}}}-1) u_{{\beta_{n}}} dx.\ee
Invoking Fourier transform technique as in Lemma \ref{n2} we have
\be \int_{\R} (u_{\beta_{n}}'')^2-\beta_{n}^2 \int_{\R} (u_{\beta_{n}}')^2\geq -\frac{\beta_{n}^4}{4} \int_{\R} u_{\beta_{n}}^2.\ee
Hence we have
\be -\frac{\beta_{n}^4}{4} \int_{\R} u_{\beta_{n}}^2\leq - \int_{\R} (e^{u_{\beta_{n}}}-1) u_{\beta_{n}} dx.\ee
This implies
\be \int_{\R} \bigg[\frac{\beta_{n}^4}{4}  u_{\beta_{n}}^2-  (e^{u_{\beta_{n}}}-1) u_{\beta_{n}}\bigg]dx \geq 0. \ee
As a result, there exists $x_{n}\in \R$ such that
$$\frac{\beta_{n}^4}{4}  u_{\beta_{n}}^2(x_{n})\geq  (e^{u_{\beta_{n}}(x_{n})}-1) u_{\beta_{n}}(x_{n}).$$ Note that this can only happen when
$ u_{\beta_{n}}(x_{n})<0.$ Hence
$$\frac{\beta_{n}^4}{4}\geq \frac{e^{u_{\beta_{n}}(x_{n})}-1}{u_{\beta_{n}}(x_{n})}\geq e^{u_{\beta_{n}}(x_{n})}.$$
Thus $e^{u_{\beta_{n}}(x_{n})}\leq \frac{\beta_{n}^4}{4}$ and hence $u_{\beta_{n}}(x_{n})\leq \ln\frac{\beta_{n}^4}{4}.$  Hence there exists $x_{0}\in \R$ such that $u_{\beta} (x_{0})\leq \ln \frac{\beta^4}{4}<0.$
This implies $u_{\beta}$ is a nontrivial solution of (\ref{a1}).

\section{Range of $\beta$ and Decay Estimates}
In this section, we first find explicit bound for $\beta$ so that (\ref{1.44-1}) holds, and then we prove the decay estimate.
\subsection{Estimate of $\beta$}
First, we find $a_{\star}$.   We recall the following eigenvalue  problem
\be\label{1.38} \left\{
    \begin{aligned}
   \varphi''''+ \beta^2 \varphi''&+ e^{u_{\star}}\varphi = 0 \; &&\text{in } (-a, a) \\
 \varphi(\pm a)&= \varphi'(\pm a)= 0.&&\text{ }
\end{aligned}\right.
\ee
Any even solution of (\ref{1.38}) can be written as $u(x)= A \cos \mu_1 x+ B \cos \mu_2 x$ where
$$\mu_1= \sqrt{\frac{\beta^2}{2}-\sqrt{\frac{\beta^4}{4}- e^{u_{\star}}}}$$ and
$$\mu_2= \sqrt{\frac{\beta^2}{2}+\sqrt{\frac{\beta^4}{4}- e^{u_{\star}}}}.$$
Then they must satisfy
\be\label{1.40}\mu_2 \tan \mu_2 a= \mu_1 \tan \mu_1 a.\ee
Since $\mu_1< \mu_2,$ the function $\mu\tan \mu a$ is increasing where $a\in (0, \frac{\pi}{2\mu_2}).$ Hence (\ref{1.40}) admits a solution in $(\frac{\pi}{2\mu_2}, \frac{3\pi}{2\mu_2}).$
When $e^{u_{\star}}\ll 1,$ $\mu_1$ is close to zero, then
$$\mu_2 \tan \mu_2 a\approx \mu_1 ^2 a.$$
Let $\mu_2 a= \pi+ t.$ Then from (\ref{1.40})
\be\no t \leq \tan t=\frac{\mu_1}{\mu_2}\tan\frac{\mu_1}{\mu_2} (\pi+t), \text{}~~~ t\in (0,\frac{\pi}{2}). \ee

Thus we obtain that $ a \leq a_{\star}$ where
\be a_{\star} :=  \frac{\pi}{\mu_2} + \frac{\mu_1}{\mu_2} \tan\frac{3 \mu_1}{2 \mu_2} \pi.
\ee
The condition (\ref{1.44-1}) can be checked numerically using (\ref{ustar}) to find an approximate bound for $u_{\star}$ and we find that the numerical bound for (\ref{1.44-1}) to hold if $\beta \leq \beta_{\star} \approx 0.742\cdots$.

\subsection{Decay estimates of (\ref{a1})}
Note that $u(x)\rightarrow 0$ as $x \rightarrow \pm\infty.$ Hence \be \frac{e^{u}-1}{u} \rightarrow 1\ee
Hence the limiting equation for fixed $\beta$ at infinity is given by
\be\label{lim} w''''+ \beta^2 w''+ w=0.\ee
Note that this is a linear problem and the roots of the
\be\label{chr} m^4+ \beta^2 m^2+ 1=0\ee  and hence
\be\label{chr1}\bigg(m^2+ \frac{\beta^2}{2}\bigg)^2= \bigg(1-\frac{\beta^4}{4}\bigg)i^2\ee
where $i=\sqrt{-1}.$
Define $n=m^2$ then we have
\be\label{chr2} n=-\frac{\beta^2}{2} \pm\bigg(1-\frac{\beta^4}{4}\bigg)i^2.\ee
Define $\cos 2\eta=-\frac{\beta^2}{2}.$
Then we can write (\ref{chr2}) as
$$n= (\cos 2\eta+ i \sin 2\eta )= (\cos \eta+ i \sin \eta )^2$$ where $\eta\in (\frac{\pi}{4}, \frac{\pi}{2})$ when $x \rightarrow -\infty$ and $\eta\in (- \frac{\pi}{2}, -\frac{\pi}{4})$ when $x \rightarrow +\infty$
as we are looking for decaying solutions.
This implies that the four roots of (\ref{chr1}) are precisely $m= e^{\pm i \eta}$ and $\bar{m}=e^{\pm i \bar{\eta}}.$
If
$$m = e^{i\eta}= \tau + i\de $$
where $\tau= \cos \eta$ and $\de= \sin \eta.$
Hence the general solution of (\ref{lim}) decaying at $\pm\infty$ is given by
\be\label{dec} w(x)= e^{\tau x} \cos( \de x+ b)\chi_{\{x<0\}} +  e^{-\tau x} \cos(\de x+ d)\chi_{\{x>0\}} \ee for some $b, d\in \R$, $\chi$ denotes the characteristic function.
Also note that $\tau$ depends on $\beta.$ As a result, $u$ decays exponentially for each $\beta>0.$

\section{Proof of Theorem \ref{th2}}
The ideas used in proving Theorem \ref{th1} can be readily extended to (\ref{a2}). We make a change of variable $ u-1$ in (\ref{a2}). Then the equation transforms into
\begin{equation}
  \label{aa2}
\left\{\begin{aligned}
      u''''+ \beta^2 u''+ u^{3}+ 3u^2+ 2u &= 0 &&\text{in } \R\\
      u(x) &> -2 &&\text{in } \R\\
      u&\in H^{2}(\R).
    \end{aligned}
  \right.
\end{equation}
Define $J_{\beta}: H^{2}(\R) \rightarrow \R$ as
$$J_{\beta}(u)=\frac{1}{2}\int_{\R} (u'')^2- \frac{\beta^2}{2} \int_{\R} (u')^2+ \frac{1}{4} \int_{\R} u^4+ \int_{\R} u^3+ \int_{\R} u^2 .$$

Smets-van den Berg \cite{SW}  proved that for {\em almost all} $\beta \in (0, \sqrt{8})$, problem (\ref{aa2}) has a homoclinic solution with
\begin{equation}
u>-2.
\end{equation}
Let $ u_{\star}$ be such that
\begin{equation}
u_{\star}= -2+ \frac{k_0}{\sqrt{2}} \beta^2
\end{equation}
By our assumption $ \beta^2 <\frac{\sqrt{2}}{k_0}$, we have $u_{\star} \leq -1$.
In $A= \{ u \geq u_{\star}\}$, we  have
\be
\label{ia3}
 I_{A} \geq \frac{1}{2} \int_{A} (u^{''})^2 -\frac{\beta^2}{2} \int_{A_2} (u^{'})^2 + \frac{k_0^2 \beta^4 }{8} \int_{A} u^2 \geq 0
\end{equation}
by (\ref{K1}) and (\ref{K2}).

Let $A^c=\{u\leq u_{\star}\}=\cup_{j=1}^{k}(a_{j}, b_{j})$, $k$ is finite since $u$ is homoclinic. Let $(a, b)$ be one of the intervals in $A^c$.
If $v=u-u_{\star},$ then we have
\be\label{res1} v''''+ \beta^2 v''+  u(u+1)(u+2)=0.\ee
Multiplying by $v$ and integrating (\ref{res1}) we obtain
\be \label{res2-1} -v'v''\mid_{a}^{b}+\int_{a}^{b} (v'')^2-\beta^2\int_{a}^{b} (v')^2+\int_{a}^{b} u(u+1)(u+2)(u-u_{\star}) =0.\ee
Similar to (\ref{poho}), we obtain a pointwise identity
\be\label{poho1} u'(x) u'''(x)- \frac{(u''(x))^2}{2}+\frac{\beta^2}{2} (u'(x))^2+ \frac{1}{4} u^2(x) (u(x)+2)^2=0.\ee
Integrating (\ref{poho}) we have
\be \label{poho1-10} v'v''\mid_{a}^{b}-\frac{3}{2}\int_{a}^{b} (v'')^2+\frac{\beta^2}{2}\int_{a}^{b} (v')^2+ \frac{1}{4}\int_{a}^{b}  u^2 (u+2)^2=0.\ee
Adding (\ref{res2-1}) and (\ref{poho1-10}) we obtain that
\be \label{id1-1n} \int_{a}^{b} (v'')^2+ \beta^2 (v')^2= 2\int_{a}^{b} u(u+1)(u+2) (u-u_{\star})+ \frac{1}{2} \int_{a}^{b} u^2 (u+2)^2 dx.\ee
Substituting (\ref{id1-1n}) into the energy over $(a,b)$ we have
\bea\label{rese-1} I_{\beta}\mid_{(a, b)}(u)&=& \frac{1}{2}\int_{a}^{b} (u'')^2- \frac{\beta^2}{2}\int_{a}^{b} (u')^2+ \int_{a}^{b} \frac{1}{4} u^2 (u+2)^2dx\no\\&=&\int_{a}^{b} (v'')^2- \int_{a}^{b} u(u+1)(u+2)( u-u_{\star}).\eea

Now we claim that $$I_{\beta}\mid_{(a, b)}(u)\geq 0.$$ In fact,   we have $    u+2>0, u<0, u+1 \leq u_{\star}+1 \leq 0, $ and hence
\begin{equation}
 u(u+1)(u+2) (u-{u_{\star}}) \leq 0 \ \mbox{on} \ (a, b)
\end{equation}
This implies $ I_{\beta} \mid_{(a, b)} (u) \geq 0$ and hence
\be
I_{\beta} \mid_{A} (u) \geq 0.
\ee
The rest of the proof is similar to that of Theorem \ref{th1}. We omit the details. \qed

\section*{acknowledgement} The research was partially supported by an ARC grant, an Earmarked grant from
 RGC of Hong Kong and a Focused Research Scheme from CUHK.

\end{document}